\documentclass[11pt,a4paper,reqno]{amsart}
\usepackage[utf8]{inputenc}
\usepackage[T1]{fontenc}
\usepackage{amsmath,amssymb}
\usepackage{graphicx,cleveref,interval,bbm,microtype}
\usepackage{geometry}
\usepackage{lmodern}

\theoremstyle{plain} 
\newtheorem{thm}{Theorem}[section]
\newtheorem{prop}[thm]{Proposition}

\newtheorem{lem}[thm]{Lemma}
\newtheorem{defn}[thm]{Definition} 
\newtheorem{ass}{Assumption}
\theoremstyle{remark}
\newtheorem{rem}[thm]{Remark}

\makeatletter
\newcommand{\proofpart}[2]{%
	\par
	\addvspace{\medskipamount}%
	\noindent\emph{Step #1.} #2\par\nobreak
	\@afterheading
}
\makeatother

\usepackage{enumitem}
\newlist{myenum}{enumerate}{2}
\setlist[myenum]{
	label=(\roman*),
	parsep = 1pt,
	itemsep = 5pt}

\newcommand{\ito}{It\^{o}}
\newcommand{\wto}{\rightharpoonup}
\newcommand{\E}{\mathbb{E}}
\renewcommand{\Pr}{\mathbb{P}}
\newcommand{\R}{\mathbb{R}}
\newcommand{\bb}{\mathbf{b}}
\newcommand{\bba}{\mathbf{b}^\ast}
\newcommand{\N}{\mathbb{N}}
\newcommand{\cF}{\mathcal{F}}
\newcommand{\cO}{\mathcal{O}}
\newcommand{\norm}[1]{\left\lVert#1\right\rVert}
\newcommand{\abs}[1]{\left\lvert#1\right\rvert}
\newcommand{\inner}[1]{\langle#1\rangle}
\DeclareMathOperator{\sign}{sign}
\DeclareMathOperator{\dive}{div}

\title[Existence and extinction for Stratonovich gradient noise PMEs]{Existence and extinction in finite time for Stratonovich gradient noise porous media equations}

\author[M.\ Turra]{Mattia Turra}
\address{Department of Computer Science, Universit\`a degli Studi di Verona \\ Strada Le Grazie 15, I-37134 Verona, Italy.}
\email{turra.mattia@outlook.com}

\thanks{}
\keywords{Stochastic partial differential equations; Porous media equations; Multiplicative Stratonovich gradient noise; Fast diffusion; Self-organized criticality}
\subjclass[2010]{60H15, 35K55, 76S05}

\begin{document}
	
\begin{abstract}
	We study existence and uniqueness of distributional solutions to the stochastic partial differential equation $dX - \bigl( \nu \Delta X + \Delta \psi (X) \bigr) dt = \sum_{i=1}^N \langle b_i, \nabla X \rangle \circ d\beta_i$ in $\interval[open]{0}{T} \times \cO$, with $X(0) = x(\xi)$ in $\cO$ and $X = 0$ on $\interval[open]{0}{T} \times \partial \cO$.
	Moreover, we prove extinction in finite time of the solutions in the special case of fast diffusion model and of self-organized criticality model. 
\end{abstract}	

\maketitle

\section{Introduction}

In this work we consider stochastic porous media equations with Stratonovich gradient noise. 
In particular, we deal with existence and uniqueness of a solution to such kind of equations, providing also some results concerning its asymptotic behaviour.
To be precise, let $\cO \subset \R^d$ be an open, bounded set with regular boundary and $T>0$, then we consider the following stochastic partial differential equation (SPDE) in $\interval[open]{0}{T} \times \cO$,
\begin{equation}\label{eq:intro1}
dX(t,\xi) - \bigl( \nu \Delta X (t,\xi) + \Delta \psi (X(t,\xi)) \bigr) dt = \sum_{i=1}^N \langle b_i(\xi), \nabla X(t,\xi) \rangle \circ d\beta_i(t),
\end{equation}
where $\circ$ denotes that the integration is intended in the Stratonovich sense, $\nu > 0$, $\psi \colon \R \to 2^\R$ is a maximal monotone function with polynomial growth, $b_i \colon \overline{\cO} \subset \R^d \to \R^d$ are $C^2$ functions and $\beta = (\beta_i)_{i=1,\ldots, N}$ is an $N$-dimensional Brownian motion on a given probability space. 

We provide existence of a distributional solution to \cref{eq:intro1}, essentially studying problem \eqref{eq:intro1} with $\psi$ substituted by its Yosida approximation, $\psi_\lambda$, $\lambda > 0$, which, as we will see, admits a solution for all $\lambda >0$, and then showing that the associated sequence of solutions, namely the sequence of solutions of the Yosida approximation scheme related to \cref{eq:intro1}, converges to the one we are looking for, as $\lambda$ goes to zero.
This approach has been employed also in, e.g., \cite{barbu2016stochastic} to study porous media equations with multiplicative noise.

We then treat two particular examples of \cref{eq:intro1}, namely fast diffusion and self-organized criticality, and prove an asymptotic result concerning the solution to the equation in those frameworks.
As we shall see, the solution will be zero from a certain time on with positive probability.
It is worth to mention that the method we have used to obtain the aforementioned results is based on the one developed in~\cite{barbu2009finite}, where the multiplicative noise case is studied.

Perturbing a problem by a Stratonovich noise of gradient type is useful in a wide range of applications, as in the case of image processing, see \cite{sixou2014stochastic,wang2015filtered}, where it has been proved that considering those kind of perturbations improves the solution obtained by the total variation regularization.

Some recent results regarding equations with Stratonovich gradient noise have been studied in the case of $p$-Laplacian and total variation flow drift.
For instance the reader may refer to \cite{barbu2013existence} which deals with $p$-Laplace equations in the case $p>1$, to \cite{ciotir2016nonlinear} for the case with Neumann boundary conditions, to \cite{2018arXiv180307005T} for $p$-Laplace equation and total variation flow on a $d$-dimensional torus, to \cite{barbu2017stochastic} for a distributional solution when $b_i$ is divergence-free and to \cite{munteanu2018total} for the case of total variation flow with Dirichlet boundary conditions.

\subsection{Structure of the work}

The work is organized as follows. 
We begin introducing the problem and discussing the assumptions, the definition of solution to \cref{eq:intro1} and some preliminaries in \Cref{sec:framework}.
\Cref{sec:approx} is devoted to the construction of the approximating problem and its properties.
We prove the existence and uniqueness of a solution to our problem in \Cref{sec:existence}.
Extinction in finite time for the fast diffusion model is discussed in \Cref{sec:extinction}. 
Self-organized criticality model is treated in \Cref{sec:soc}.
We conclude with some final remarks and considerations in \Cref{sec:conclusion}.

\subsection{Notation}

Let $\cO \subset \R^d$ be an open and bounded set with regular boundary~$\partial \cO$. 
Then we denote by $L^p (\cO)$, $p \in \interval{0}{\infty}$, the Banach space of all $p$-summable (equivalence classes of) functions from $\cO$ to $\R$ and by $\abs{\cdot}_p$ its corresponding norm, while we indicate by $H^k(\cO)$, $k \in \N$, the Sobolev space of functions in $L^2(\cO)$ whose distributional derivatives of order less than $k$ belong to $L^2$. 
$H^1_0(\cO)$ is the set of $H^1$ function vanishing on $\partial \cO$, its corresponding norm is given by
\[ \norm{u}_1 \doteq \left( \int_{\R^d} \abs{\nabla u}_{\R^d}^2 \right)^{1/2}. \]
$H^{-1}(\cO)$ is the dual of $H^1_0(\cO)$ and its norm is denoted by $\norm{\cdot}_{-1}$, $\inner{\cdot, \cdot}_{-1}$ being its inner product.
We indicate by $\inner{\cdot,\cdot}$, or sometimes simply by $\cdot$, the scalar product on $\R^N$.

\section{Framework} \label{sec:framework}

Let $\cO \subset \R^d$ be a open and bounded set with smooth boundary $\partial \cO$.
We aim at providing existence and uniqueness of a solution to the following nonlinear SPDE with Stratonovich gradient noise for $X \colon \Omega \times \interval{0}{T} \times \cO\to \R$,
\begin{equation}\label{eq:spde strato}
\begin{cases}
\displaystyle{dX(t) - \bigl( \nu \Delta X (t) + \Delta \psi (X(t)) \bigr) dt = \sum_{i=1}^N \langle b_i, \nabla X(t) \rangle \circ d\beta_i(t)}, & \text{in } \interval[open]{0}{T} \times \cO, \\[5pt]
X(0,\xi) = x(\xi), & \text{in } \cO, \\[5pt]
X(t,\xi) = 0, & \text{on } \interval[open]{0}{T} \times \partial \cO,
\end{cases}
\end{equation}
where $\nu > 0$, $\bb = (b_1,\ldots,b_N) \colon \overline{\cO} \to \R^{N\times d}$, $\psi\colon \R \to 2^\R$ is a (possibly multivalued) map and $\beta = (\beta_i)_{i=1,\ldots,N}$ is an $N$-dimensional Brownian motion on a filtered probability space $(\Omega, \cF, (\cF_t)_{t\ge 0}, \Pr)$.
For the sake of simplicity, we will often omit to write explicitly the dependence of $X$ on $(\omega,t,\xi) \in \Omega \times \interval{0}{T} \times \cO$ as well as the dependence of $\bb$ on $\xi \in \overline{\cO}$.

\Cref{eq:spde strato} can be equivalently written as the following SPDE in the \ito\ sense
\begin{equation}\label{eq:spde ito}
\begin{cases}
\begin{aligned} 
dX(t) - \Bigl( \nu \Delta X(t) + \Delta \psi (X(t)) + \frac{1}{2} \dive & (\bba \, \bb \, \nabla X(t)) \Bigr) dt \, = \\ & = \langle \bb \nabla X(t), d\beta (t) \rangle , 
\end{aligned} 
& \text{in } \interval[open]{0}{T} \times \cO, \\[4pt]
X(0,\xi) = x(\xi), & \text{in } \cO, \\[4pt]
X(t,\xi) = 0, & \text{on } \interval[open]{0}{T}\times \partial \cO,
\end{cases}
\end{equation}
where $\bba$ is the transpose of $\bb$.

Notice that we can write
\begin{align}
\dive (\bba \bb \, \nabla X) & = \sum_{k=1}^d \frac{\partial}{\partial \xi_k} \left( \sum_{j=1}^d A_{kj} (\xi) \, \frac{\partial X(t,\xi)}{\partial \xi_j} \right),\label{eq:dive-rewritten} \\[6pt]
\bb \, \nabla X & = \sum_{i=1}^d \begin{pmatrix}
b_{1i} (\xi) \\ \vdots \\ b_{Ni} (\xi)
\end{pmatrix} \frac{\partial X(t,\xi)}{\partial \xi_i}, \notag 
\end{align}
where
\begin{equation} \label{eq:a funct}
A_{kj} (\xi) = \sum_{i=1}^N b_{ik} (\xi) \, b_{ij} (\xi), \qquad k,j = 1,\ldots,d. 
\end{equation}  

Before proceeding further we provide the following result, which will be useful in the proof of the existence of a solution to \cref{eq:spde ito}.
\begin{lem} \label{thm:dive-estimates}
	Let $\bb$ be defined as above and assume that $b_{ij} \in C^2(\overline{\cO})$. 
	Then there exist $\tilde{C} > 0$, depending on $\cO$, and $\gamma > 0$, depending on $\bb$, such that
	\begin{equation*}
	\abs{(-\Delta)^{-1} \dive(\bb^\ast \bb \nabla u)}_2 \le \tilde{C} \gamma \abs{u}_2, \quad \text{for every } u \in L^2(\cO).
	\end{equation*}
\end{lem}

\begin{proof}
	Let $z \doteq (-\Delta)^{-1} \dive(\bb^\ast \bb \nabla u)$ in $\cO$, that is, equivalently,
	\begin{equation}\label{eq:dive-est-proof-1}
	\begin{cases}
	- \Delta z = \dive (\bb^\ast \bb \nabla u), & \text{in } \cO, \\
	z = 0, & \text{in } \partial \cO. 
	\end{cases}
	\end{equation}
	Let $f \in L^2(\cO)$ and $v \in H^1_0(\cO) \cap H^2(\cO)$ be the solution to
	\begin{equation}\label{eq:dive-est-proof-2}
	\begin{cases}
	- \Delta v = f, & \text{in } \cO, \\
	v = 0, & \text{on } \partial \cO.
	\end{cases}
	\end{equation}
	Multiplying the first equation in \eqref{eq:dive-est-proof-1} by $v$ and then integrating we have, by Green's formula and~\eqref{eq:dive-est-proof-2},
	\[ \int_\cO z  f \, d\xi = \int_\cO \dive (\bb^\ast \bb \nabla u) \, v \, d\xi = \int_\cO u \, \dive (\bb^\ast \bb \nabla v) \, d \xi, \]
	therefore, by \eqref{eq:dive-rewritten}, we get
	\[ 
	\begin{split}
	\abs{\inner{z,f}_2} & \le \abs{u}_2 \abs{\dive (\bb^\ast \bb \nabla v)}_2 \\ 
	& \le \abs{u}_2  \sum_{k,j = 1}^d \left( \abs{A_{kj}}_\infty \bigl\lvert {D_{jk}^2 v}\bigr\rvert_2 + \abs{D_k A_{kj}}_\infty \abs{D_j v}_2  \right) \\
	& \le C \gamma \abs{u}_2 \left( \norm{v}_{H^2(\cO)} + \norm{v}_{H^1_0(\cO)} \right),
	\end{split}
	\]
	where $C = C(\cO)$ and 
	\begin{equation*} \label{eq:gammab}
	\gamma = \gamma (\bb) \doteq \max \left\{ \abs{A_{kj}}_\infty + \abs{D_k A_{kj}}_\infty \colon \, k,j \in \{ 1, \ldots, d\} \right\}.
	\end{equation*} 
	By \cref{eq:dive-est-proof-2} we have
	\[ \norm{v}_{H^1_0(\cO)} + \norm{v}_{H^2(\cO)} \le K \abs{f}_2, \]
	for some $K$ depending on $\cO$, so that
	\[ \abs{\inner{z,f}_2} \le \tilde{C} \gamma \abs{u}_2 \abs{f}_2 , \quad \text{for every } f \in L^2(\cO), \]
	where $\tilde{C} = C  K$ depends on $\cO$.
	Hence $\abs{z}_2 \le \tilde{C} \gamma \abs{u}_2$, for every $u \in L^2(\cO)$, which concludes the proof.
\end{proof}

We assume that the following hypotheses on $\psi$, $\bb$ and $\nu$ hold.
\begin{ass} \label{ass: on b}
	The following hypotheses hold:
	\begin{myenum}
		\item \label{H1} $\psi \colon \R \to 2^\R$ is maximal monotone with $0 \in \psi(0)$.
		\item \label{H2} There exists $C > 0$ and $m \ge 0$ such that
		\[ \sup \left\{ \abs{\theta} \colon \theta \in \psi(r) \right\} \le C (1+\abs{r}^m), \quad \text{for every } r \in \R. \]
		Moreover, we assume that $m \le d/(d-2)$ if $d \ge 3$.
		\item \label{H3} The functions $A_{kj}$ defined in \eqref{eq:a funct} are bounded for any $k,j = 1,\ldots,d$.
		\item $b_i \in C^2 \bigl(\overline{\cO};\R^d \bigr)$ for every $i = 1, \ldots, N$ and 
		\begin{equation} \label{eq:hp-on-b-nu}
		\tilde{C} \gamma (\bb) + \abs{\bb}_\infty^2 \le 2 \nu, 
		\end{equation}
		where $\tilde{C}$ and $\gamma$ are as in \Cref{thm:dive-estimates}.
	\end{myenum}
\end{ass} 

Notice that condition \eqref{eq:hp-on-b-nu} tells us that the nearer $\nu$ is to $0$, the stricter the condition on the norm of $\bb$ is.
In particular, the case $\nu = 0$ implies $\bb \equiv 0$, reducing \cref{eq:spde ito} to the (deterministic) PDE $\partial_t X = \Delta \psi(X)$.

The solution we are looking for is to be intended in the following sense.
\begin{defn}\label{def:strong solution}
	A \emph{solution} to \cref{eq:spde ito} in $\interval{0}{T}$ is an $(\cF_t)_{T\ge 0}$-adapted stochastic process $X$ such that
	\begin{myenum}
		\item $X \in L^2\bigl(\Omega; L^\infty\bigl(\interval{0}{T};H^{-1}(\cO)\bigr)\bigr) \cap L^2\bigl(\Omega \times \interval{0}{T};H^1_0(\cO)\bigr)$,
		\item there exists a process $\eta \in L^2(\Omega \times \interval[open]{0}{T} \times \cO)$ such that $\eta \in \psi(X)$ a.e., and
		\begin{multline} \label{eq:distr-sol-strato}
		\begin{aligned} 
		\inner{X(t), & f_j}_{-1} = \inner{x, f_j}_{-1} - \nu \int_0^t \inner{X(s) , f_j}_{2} ds - \int_0^t \inner{\eta(s),f_j}_2 ds \\[3pt]
		& + \frac{1}{2} \int_0^t \inner{\dive (\bb^\ast \bb \nabla X(s), f_j}_{-1} ds + \int_0^t \inner{ \bb \nabla X(s), f_j}_{-1} d\beta (s),
		\end{aligned} \\[3pt]
		\text{for every } j \in \N, \, t \in \interval{0}{T}, \, \Pr\text{-a.s.},
		\end{multline}
		where $(f_j)_{j \in \N}$ is an orthonormal basis for $-\Delta$ in $H^{-1}$.
		\item $X$ is pathwise continuous from $\interval{0}{T}$ to $H^{-1}(\cO)$.
	\end{myenum}
\end{defn}
A solution of this type is also referred to as \emph{distributional solution} since we can equivalently write \cref{eq:distr-sol-strato} as
\begin{equation*} 
X(t) = x - \nu \int_0^t \Delta X(s) \, ds - \Delta\! \int_0^t \eta(s) \, ds
+ \frac{1}{2} \int_0^t \dive (\bb^\ast \bb \nabla X(s)) \, ds + \int_0^t \bb \nabla X(s) \cdot  d\beta(s), 
\end{equation*}
where $\Delta \colon H^1_0(\cO) \to H^{-1}(\cO)$ is taken in the sense of distributions on $\cO$.

\section{The approximating problem} \label{sec:approx}

Under \Cref{ass: on b} we define, for every $\lambda > 0$, the \emph{resolvent} and the \emph{Yosida approximation} of $\psi$,
\begin{equation*}
J_\lambda \doteq (I+\lambda \psi)^{-1}, \qquad 
\psi_\lambda \doteq \frac{1}{\lambda} (I-J_\lambda),
\end{equation*}
respectively, which are known to be Lipschitz-continuous, see, e.g., \cite{barbu2010nonlinear}.
We shall consider thus the following \emph{approximating problem}, for $\lambda > 0$,
\begin{equation}\label{eq:spde approx}
\begin{cases}
\begin{aligned} dX_\lambda - \Bigl( \nu \Delta X_\lambda + \Delta \psi_\lambda (X_\lambda) + \frac{1}{2} \dive (\bba \, \bb  & \, \nabla X_\lambda) \Bigr) dt = \\ & = \langle \bb \nabla X_\lambda, d\beta \rangle , \end{aligned}  & \text{in } \interval[open]{0}{T} \times \cO, \\
X_\lambda(0,\xi) = x(\xi), & \text{in } \cO, \\
X_\lambda(t,\xi) = 0, & \text{on } \interval[open]{0}{T} \times \partial \cO.
\end{cases}
\end{equation}

We will need the following result.
\begin{lem} \label{thm:growth-yosida}
	We have, for every $r \in \R$ and $\lambda > 0$,
	\[ \abs{\psi_\lambda (r)} \le C(1+\abs{r}^m).  \]
\end{lem}
\begin{proof}
	It holds, for every $r \in \R$,
	\begin{equation*} 
	\abs{\psi_\lambda (r)} \le \sup \left\{ \abs{\theta} \colon \theta \in \psi (J_\lambda (r)) \right\} \le C (1+\abs{J_\lambda (r)}^m) \le C(1+\abs{r}^m). \qedhere 
	\end{equation*} 
\end{proof}

Moreover, one can also see that $r\mapsto \psi_\lambda (r) + \nu r$ is strictly monotonically increasing, bounded by $C (1+\abs{r}^m)$ and $(\psi_\lambda (r) + \nu r) r \ge \nu \abs{r}^2$ for all $r \in \R$.

We have, thus, the following existence result, which is a consequence of Krylov and Rozovskii Theorem, see \cite{krylov1981stochastic} or the more recent book~\cite{liu2015stochastic}.

\begin{prop}\label{thm:approximating-existence-uniqueness}
	Suppose \Cref{ass: on b} holds. Then \cref{eq:spde approx} admits a unique variational solution.
\end{prop}

We conclude this section with the following result.  
\begin{lem} \label{thm:grad-estimates}
	Let $X_\lambda$ be a solution to \cref{eq:spde approx}. 
	Then there exist $C_1, C_2 > 0$, depending on $\nu$ and on the $L^2$-norm of the initial condition $x$ but not on the parameter $\lambda$, such that
	\[ \E \abs{X_\lambda (t)}_2^2 + C_1 \E \int_0^t \abs{\nabla X_\lambda (s)}_2^2 ds \le C_2. \]
\end{lem}

\begin{proof}
	Let $X_\lambda$ be the solution to \cref{eq:spde approx}, then by \ito's formula in $L^2$ we have 
	\[ d \abs{ X_\lambda (t)}_2^2 - 2 \inner{X_\lambda (t), \nu \Delta X_\lambda (t) + \Delta \psi_\lambda (X_\lambda(t))}_{2} dt = 2 \sum_{i=1}^N \inner{X_\lambda (t), b_i \cdot \nabla X_\lambda (t) \, d \beta_i (t)}_2. \]
	Hence,
	\begin{multline*}
	\abs{ X_\lambda (t)}_2^2 - 2 \int_0^t \inner{X_\lambda (s), \nu \Delta X_\lambda (s) + \Delta \psi_\lambda (X_\lambda(s))}_{2}  ds = \\
	= 2 \sum_{i=1}^N \int_{0}^t \inner{X_\lambda (s), b_i \cdot \nabla X_\lambda (s) \, d \beta_i (s)}_2,
	\end{multline*}
	which gives
	\begin{multline*}
	\abs{ X_\lambda (t)}_2^2 + 2 \nu \int_0^t \abs{ \nabla X_\lambda (s) }_2^2 ds + 2 \int_0^t \int_{\cO} \psi_\lambda'(X_\lambda(s)) \abs{\nabla X_\lambda (s)}^2 d\xi \, ds = \\
	= \abs{x}_2^2 + 2 \sum_{i=1}^N \int_{0}^t \inner{X_\lambda (s), b_i \cdot \nabla X_\lambda (s) \, d \beta_i (s)}_2.
	\end{multline*}
	Then, taking the expectations,
	\begin{equation*}
	\E \abs{ X_\lambda (t)}_2^2 + 2 \nu \E \int_0^t \abs{ \nabla X_\lambda (s) }_2^2 ds + 2 \E \int_0^t \int_{\cO} \psi_\lambda'(X_\lambda(s)) \abs{\nabla X_\lambda (s)}^2 d\xi \, ds 
	= \abs{x}_2^2.
	\end{equation*}
	The maximal monotonicity of $r\mapsto \psi(r)$ yields the result.
\end{proof}

\section{Existence and uniqueness of the solution} \label{sec:existence}

This section is devoted to the proof of the main result of the work, namely the existence and uniqueness of a solution to \cref{eq:spde ito}. 

\begin{thm}\label{thm:existence}
	If \Cref{ass: on b} holds and $x \in L^2(\cO)$, then \cref{eq:spde ito} admits a unique solution in the sense of \Cref{def:strong solution}.
\end{thm}

\begin{proof}
	As concerns existence, we now prove that the sequence of solutions to \cref{eq:spde approx}, $(X_\lambda)_{\lambda>0}$, is a Cauchy sequence in $L^2(\Omega; L^\infty(\interval{0}{T};H^{-1}(\cO)))$ as $\lambda \to 0$.
	Consider $X_\lambda$ and $X_\mu$, with $\lambda, \mu > 0$, then by \ito's formula in $H^{-1}$ we get
	\[ 
	\begin{split}
	\norm{X_\lambda(t) - X_\mu(t)}_{-1}^2 = \, & 2 \nu \int_0^t \inner{\Delta (X_\lambda(s) - X_\mu(s) ) , X_\lambda (s) - X_\mu (s) }_{-1} ds \\ 
	& + 2 \int_0^t \inner{\Delta \psi_{\lambda} (X_\lambda(s)) - \Delta \psi_\mu (X_\mu(s)), X_\lambda (s) - X_\mu (s)}_{-1} ds \\
	& + \int_0^t \inner{\dive (\bb^\ast \bb \nabla (X_\lambda(s) - X_\mu(s))), X_\lambda(s) - X_\mu(s)}_{-1} ds \\
	& + \int_0^t \norm{\bb \nabla (X_\lambda(s) - X_\mu(s))}_{-1}^2 ds \\
	& + 2 \sum_{i=1}^N \int_0^t \inner{ X_\lambda (s) - X_\mu (s) , b_i \cdot \nabla (X_\lambda (s) - X_\mu (s)) \, d \beta_i(s)}_{-1},
	\end{split} 
	\] 
	which can be written as
	\begin{equation} \label{eq:norm-diff-h1} 
	\begin{split} 
	\norm{X_\lambda(t) - X_\mu(t)}_{-1}^2 = \, & - 2 \nu \int_0^t \abs{ X_\lambda (s) - X_\mu (s) }^2_{2} ds \\ 
	& - 2 \int_0^t \inner{\psi_{\lambda} (X_\lambda(s)) - \psi_\mu (X_\mu(s)), X_\lambda (s) - X_\mu (s)}_{2} ds \\
	& + \int_0^t \inner{\dive (\bb^\ast \bb \nabla (X_\lambda(s) - X_\mu(s))), X_\lambda(s) - X_\mu(s)}_{-1} ds \\
	& + \int_0^t \norm{\bb \nabla (X_\lambda(s) - X_\mu(s))}_{-1}^2 ds \\
	& + 2 \sum_{i=1}^N \int_0^t \inner{ X_\lambda (s) - X_\mu (s) , b_i \cdot \nabla (X_\lambda (s) - X_\mu (s)) \, d \beta_i(s)}_{-1} .
	\end{split}
	\end{equation} 
	
	We shall now provide some estimates of the right-hand side terms in \cref{eq:norm-diff-h1}. Therefore we divide the remaining part of the proof in different steps.
	
	\proofpart{1}{Estimate of $ \int_0^t \inner{\dive (\bb^\ast \bb \nabla (X_\lambda(s) - X_\mu(s))), X_\lambda(s) - X_\mu(s)}_{-1} ds $.}
	By Cauchy-Schwartz inequality we have
	\begin{equation*}
	\begin{split}  
	\abs{\inner{\dive (\bb^\ast \bb \nabla (X_\lambda - X_\mu)), X_\lambda - X_\mu}_{-1} } & \le \abs{\inner{(-\Delta)^{-1} \dive (\bb^\ast \bb \nabla (X_\lambda - X_\mu)), X_\lambda - X_\mu}_{2} } \\
	& \le \abs{(-\Delta)^{-1} \dive (\bb^\ast \bb \nabla (X_\lambda - X_\mu))}_2 \abs{X_\lambda - X_\mu}_2.
	\end{split}
	\end{equation*}
	Exploiting now \Cref{thm:dive-estimates} with $u = X_\lambda - X_\mu$ we get
	\begin{equation*}
	\abs{(-\Delta)^{-1} \dive(\bb^\ast \bb \nabla (X_\lambda - X_\mu))}_2 \le \tilde{C} \gamma \abs{X_\lambda - X_\mu}_2,
	\end{equation*}
	which yields 
	\begin{equation}\label{eq:step1}
	\abs{\inner{\dive (\bb^\ast \bb \nabla (X_\lambda - X_\mu)), X_\lambda - X_\mu}_{-1} } \le \tilde{C} \gamma \abs{X_\lambda - X_\mu}_2^2.
	\end{equation}
	
	\proofpart{2}{Estimate of $\int_0^t \norm{\bb \nabla (X_\lambda(s) - X_\mu(s))}_{-1}^2 ds$.}
	The integrand can be equivalently written as 
	\[ \norm{\bb \nabla (X_\lambda - X_\mu)}_{-1}^2 = \norm{ \bigl( b_j \cdot \nabla (X_\lambda - X_\mu) \bigr)_{j = 1, \ldots, N} }_{-1}^2. \]
	In particular, we can rewrite, for all $j = 1,\ldots,N$,
	\[ b_j \cdot \nabla (X_\lambda - X_\mu) = \dive (b_j (X_\lambda -X_\mu )) - \dive (b_j) (X_\lambda - X_\mu), \]
	so that
	\[ \norm{\bb \nabla (X_\lambda - X_\mu)}_{-1}^2  = \norm{(\dive (b_j (X_\lambda -X_\mu )) - \dive (b_j) (X_\lambda - X_\mu))_{j=1,\ldots,N}}_{-1}^2 . 
	\]
	Hence we get
	\begin{equation} \label{eq:estimate1}
	\norm{\bb \nabla (X_\lambda - X_\mu)}_{-1}^2 \le \abs{\bb}_\infty^2 \abs{X_\lambda - X_\mu}_{2}^2 + \abs{(\dive b_j)_j }_\infty^2 \norm{X_\lambda - X_\mu}_{-1}^2.
	\end{equation} 
	
	\proofpart{3}{Estimate of $- 2 \, \E \int_0^t \inner{\psi_{\lambda} (X_\lambda(s)) - \psi_\mu (X_\mu(s)), X_\lambda (s) - X_\mu (s)}_{2} ds$.}
	We have, for every $\lambda, \mu >0$,
	\begin{equation}  \label{eq:step3}
	- 2 \, \E \int_0^t \inner{\psi_{\lambda} (X_\lambda(s)) - \psi_\mu (X_\mu(s)), X_\lambda (s) - X_\mu (s)}_{2} ds \le K (\lambda + \mu).
	\end{equation} 
	Indeed, we can rewrite
	\[ 
	\begin{split} 
	X_\lambda - X_\mu & {} = X_\lambda - J_\lambda (X_\lambda) + J_\lambda (X_\lambda) - J_\mu (X_\mu) + J_\mu (X_\mu) - X_\mu \\
	& = \lambda \psi_\lambda (X_\lambda) +  J_\lambda (X_\lambda) - J_\mu (X_\mu) - \mu \psi_\mu (X_\mu).
	\end{split} 
	\]
	Thus, since we have $\inner{J_\lambda (X_\lambda)-J_\mu(X_\mu), \psi_\lambda (X_\lambda)-\psi_\mu(X_\mu)}_2 \ge 0$ (see, e.g., \cite[Prop.~2.3]{barbu2010nonlinear}),
	\begin{multline*} 
	\int_0^t \inner{X_\lambda - X_\mu , \psi_\lambda (X_\lambda) - \psi_\mu (X_\mu)}_2 \,  ds \ge \lambda \, \int_0^t \abs{\psi_\lambda (X_\lambda)}_2^2 \, ds \\
	+ \mu \, \int_0^t \abs{\psi_\mu (X_\mu)}_2^2 \, ds 
	- (\lambda + \mu) \, \int_0^t \abs{\psi_\lambda(X_\lambda)}_2 \abs{\psi_\lambda(X_\mu)}_2 \, ds.
	\end{multline*} 
	By Cauchy's inequality with epsilon we have
	\begin{multline*} 
	(\lambda + \mu) \abs{\psi_\lambda(X_\lambda)}_2 \abs{\psi_\lambda(X_\mu)}_2 \le \\ \le \lambda \left( \abs{\psi_\lambda(X_\lambda)}_2^2 + \frac{1}{4} \abs{\psi_\mu (X_\mu)}_2^2 \right) 
	+ \mu \left( \abs{ \psi_\mu (X_\mu)}_2^2 + \frac{1}{4} \abs{\psi_\lambda (X_\lambda)}_2^2 \right),
	\end{multline*}
	and so,
	\[
	\begin{split} 
	\int_0^t \inner{X_\lambda - X_\mu , \psi_\lambda (X_\lambda) - \psi_\mu (X_\mu)}_2 \,  ds \ge {} & - \frac{1}{4} \int_0^t (\mu \abs{\psi_\lambda (X_\lambda)}_2^2 + \lambda \abs{\psi_\mu (X_\mu)}_2^2) \, ds.
	\end{split} 
	\]
	
	Recalling \Cref{ass: on b}(ii) and exploiting \Cref{thm:growth-yosida}, \Cref{thm:grad-estimates} and Sobolev embedding Theorem, we have
	\begin{equation} \label{eq:est-psi-l2}
	\E\int_0^t \abs{\psi_\lambda (X_\lambda)}_2^2 ds \le C_1 \, \E \int_0^t \int_{\cO} \abs{X_\lambda}^{2m} d\xi \, ds	\le C_2,
	\end{equation} 
	so that
	\[
	\E \int_0^t \inner{X_\lambda - X_\mu , \psi_\lambda (X_\lambda) - \psi_\mu (X_\mu)}_2 \,  ds \ge - \frac{t C_2}{4} (\lambda + \mu ) \ge - \frac{T \, C_2}{4} (\lambda + \mu ),
	\]
	which implies the claim, for every $\mu, \lambda > 0$.
	
	\proofpart{4}{Proof that $(X_\lambda)_\lambda$ is a Cauchy sequence in $L^2(\Omega; L^\infty(\interval{0}{T};H^{-1}(\cO)))$, as $\lambda \to 0$.}
	Consider \cref{eq:norm-diff-h1}, then we can exploit the previous steps of the proof to get, by \cref{eq:step1,eq:estimate1},
	\[ 
	\begin{split} 
	\norm{X_\lambda(t) - X_\mu(t)}_{-1}^2 + & \,  2 \nu \int_0^t \abs{ X_\lambda (s) - X_\mu (s) }^2_{2} ds \le \\ 
	\le & - 2 \int_0^t \inner{\psi_{\lambda} (X_\lambda(s)) - \psi_\mu (X_\mu(s)), X_\lambda (s) - X_\mu (s)}_{2} ds \\
	& + (\tilde{C} \gamma +\abs{\bb}_\infty^2) \int_0^t \abs{X_\lambda(s) - X_\mu(s)}_2^2 ds \\
	& + \abs{(\dive b_j)_j }_\infty^2 \int_0^t \norm{X_\lambda (s) - X_\mu (s)}_{-1}^2 ds \\
	& + 2 \sum_{i=1}^N \int_0^t \inner{ X_\lambda (s) - X_\mu (s) , b_i \cdot \nabla (X_\lambda (s) - X_\mu (s)) \, d \beta_i(s)}_{-1} .
	\end{split}
	\]
	Recalling \eqref{eq:hp-on-b-nu}, and by Burkholder-Davis-Gundy inequality, we have
	\[ 
	\begin{split} 
	\E \sup_{s \in \interval{0}{t}} \norm{X_\lambda(s) - X_\mu(s)}_{-1}^2
	\le & - 2 \, \E \int_0^t \inner{\psi_{\lambda} (X_\lambda(s)) - \psi_\mu (X_\mu(s)), X_\lambda (s) - X_\mu (s)}_{2} ds \\
	& + \abs{(\dive b_j)_j }_\infty^2 \E \int_0^t \norm{X_\lambda (s) - X_\mu (s)}_{-1}^2 ds \\
	& + C_3 \, \E \int_0^t \norm{X_\lambda (s) - X_\mu (s)}_{-1}^2 ds.
	\end{split}
	\]
	Now estimate \eqref{eq:step3} gives
	\[ 
	\begin{split} 
	\E \sup_{s \in \interval{0}{t}} \norm{X_\lambda(s) - X_\mu(s)}_{-1}^2 
	\le & \, K(\lambda + \mu) + C_4 \, \E \int_0^t \norm{X_\lambda (s) - X_\mu (s)}_{-1}^2 ds,
	\end{split}
	\]
	where $C_4 \doteq \abs{(\dive b_j)_j}_\infty^2 + C_3$. 
	Applying Gronwall's inequality we get the result.
	
	Since $(X_\lambda)_\lambda$ is Cauchy in $L^2(\Omega; C(\interval{0}{T};H^{-1}(\cO)))$ and by \cref{eq:est-psi-l2}, we know that, as $\lambda \to 0$,
	\begin{align*}
	& X_\lambda \to X, && \text{in } L^2(\Omega; C(\interval{0}{T};H^{-1}(\cO))), \\
	& X_\lambda \to X, && \text{in } L^2(\Omega \times \interval[open]{0}{T}\times \cO), \\
	& \psi_\lambda (X_\lambda) \wto \eta, && \text{weakly in } L^2(\Omega \times \interval[open]{0}{T} \times \cO),
	\end{align*}
	and $\eta \in \psi (X)$ by the maximal monotonicity of $y \mapsto \psi(y)$ in $L^2(\Omega \times \interval[open]{0}{T} \times \cO)$, indeed $\psi_\lambda (X_\lambda) = \psi ((1+\lambda \psi)^{-1} X_\lambda)$ and $(1+\lambda \psi)^{-1}X_\lambda \to X$ in $L^2(\Omega \times \interval[open]{0}{T} \times \cO)$ as $\lambda \to 0^+$. On the other hand $\psi_\lambda (X_\lambda) = \psi((1+\lambda \psi)^{-1}X_\lambda) \wto \eta$, and so $\eta \in \psi (X)$ a.e.~on $(\Omega \times \interval[open]{0}{T} \times \cO)$, which proves the existence of a solution.
	
	\proofpart{5}{Uniqueness.}
	Consider two solution $X$ and $Y$ of \cref{eq:spde ito}, then, by \ito's formula in $H^{-1}$ and exploiting the estimates in the existence proof, we~have
	\begin{equation*}
	\begin{split} 
	\norm{X(t) - Y(t)}_{-1}^2 & + 2 \nu \int_0^t \abs{ X (s) - Y (s) }^2_{2} ds \le \\ 
	\le & - 2 \int_0^t \inner{\psi (X(s)) - \psi (Y(s)), X (s) - Y (s)}_{2} ds \\
	& + (\tilde{C} \gamma +\abs{\bb}_\infty^2) \int_0^t \abs{X(s) - Y(s)}_2^2 ds \\
	& + \abs{(\dive b_j)_j }_\infty^2 \int_0^t \norm{X (s) - Y (s)}_{-1}^2 ds \\
	& + 2 \sum_{i=1}^N \int_0^t \inner{ X (s) - Y (s) , b_i \cdot \nabla (X (s) - Y (s)) \, d \beta_i(s)}_{-1} .
	\end{split}
	\end{equation*}
	Now, taking expectation and recalling the monotonicity of $\psi$, we obtain
	\begin{multline*}
	\E \norm{X(t) - Y(t)}_{-1}^2 + \left(2 \nu - (\tilde{C} \gamma +\abs{\bb}_\infty^2)\right) \E \int_0^t \abs{ X (s) - Y (s) }^2_{2} ds \le \\ 
	\le \abs{(\dive b_j)_j }_\infty^2 \E \int_0^t \norm{X (s) - Y (s)}_{-1}^2 ds.
	\end{multline*}
	Hence, by \eqref{eq:hp-on-b-nu} we have
	\[ \E \norm{X(t) - Y(t)}_{-1}^2 \le\abs{(\dive b_j)_j }_\infty^2 \int_0^t \E \norm{X (s) - Y (s)}_{-1}^2 ds,  \]
	and Gronwall's inequality yields the result.
\end{proof}

\section{Extinction in finite time for the fast diffusion model} \label{sec:extinction}

Porous media equations of the type
\[ \frac{\partial u}{\partial t} - \Delta \psi(u) = f \]
were first used to describe the dynamics of the flow in a porous medium, see, e.g., \cite{leibenzon1929motion,muskat1937flow}. 
Indeed, the standard model of diffusion of a gas through a porous media is that where
\[ \psi (r) = \rho \abs{r}^{m-1} r, \qquad \text{for every } r \in \R, \]
with $\rho > 0$ and $m > 1$, which is the so-called \emph{slow diffusion} model. 
More generally, one can consider the case of a continuous monotone function satisfying
\[ \rho \abs{r}^{m+1} \le r \psi(r) \le a_1 \abs{r}^{q+1} + a_2 r, \quad \text{for every } r \in \R, \]
for $q > m > 1$, and $\rho, a_1 > 0$.

The case $m \in \interval[open]{0}{1}$, which we are concerned here, is that of the \emph{fast diffusion} model.
This model is relevant in the description of plasma physics, the kinetic theory of gas or fluid transportation in porous media, as suggested in, e.g.,~\cite{berryman1978nonlinear,berryman1983asymptotic}.

The reader is referred to \cite{vazquez2006smoothing,vazquez2007porous} for a complete treatment of porous media equations.

A general feature of the fast diffusion case is that it models diffusion processes with a fast speed of mass transportation and this is one reason why the process terminates within finite time with positive probability. 
This is, in fact, what we are going to show in this section for the Stratonovich gradient noise case. 
The result has been proved for the case of linear multiplicative noise, see \cite{barbu2009finite,barbu2012finite}. 
The approach used in the following is the same as the ones used in those works and in \cite[Ch.~3.7]{barbu2016stochastic}.

So, from now on, we will focus on the fast diffusion model and we will work under the following conditions.
\begin{ass}\label{ass:extinction}
	The following hypotheses hold:
	\begin{myenum}
		\item $\psi$ is as in the fast diffusion model, i.e., 
		\begin{equation}\label{eq:fast-diffusion-psi}
		\psi(r) \doteq \rho \abs{r}^{m-1} r, \quad \text{for every } r \in \R, \text{ with } m \in \interval[open]{0}{1}, \rho > 0.
		\end{equation}
		\item \Cref{ass: on b} (iii)--(iv) hold.
		\item We have
		\begin{equation}\label{eq:hp-on-m-d}
		1 \le d < \frac{2(1+m)}{1-m}.
		\end{equation}
	\end{myenum}
\end{ass}

\begin{rem}
	The function $\psi$ defined in \eqref{eq:fast-diffusion-psi} satisfies (i)--(ii) in \Cref{ass: on b} and so \cref{eq:spde ito} admits a unique solution, according to \Cref{thm:existence}.
\end{rem}

The meaning of assumption \eqref{eq:hp-on-m-d} will be clear after the following lemma, which gives us an estimate on $\norm{X_\lambda(t)}_{-1}^{1-m}$, where $X_\lambda$ is the solution to the approximating problem~\eqref{eq:spde approx} for~$\lambda > 0$.
\begin{lem}
	Suppose \Cref{ass:extinction} holds. 
	Then there exists $K_m>0$ such that, for every $\lambda > 0$ and $0 \le r \le t$,
	\begin{equation}\label{eq:extinction-estimate}
	\begin{split}
	& \norm{X_\lambda (t)}_{-1}^{1-m} + (1-m) \rho \int_r^t e^{K_m (t-s)} \norm{X_\lambda(s)}_{-1}^{-m-1} \abs{X_\lambda(s)}_{m+1}^{m+1} \mathbbm{1}_{\norm{X_\lambda(s)}_{-1} > 0}(s) \, ds \\
	& \le \, e^{K_m (t-r)} \norm{X_\lambda (r)}_{-1}^{1-m} \\
	& + (1-m) \int_r^t e^{K_m (t-s)} \norm{X_\lambda(s)}_{-1}^{-m-1} \inner{X_\lambda(s), \bb \nabla X_\lambda(s) \cdot \left( \mathbbm{1}_{\norm{X_\lambda(s)}_{-1} > 0}(s) \, d \beta (s) \right)}_{-1}.
	\end{split}
	\end{equation}
\end{lem}

\begin{proof}
	In order to get the estimate on $\norm{X_\lambda(t)}_{-1}^{1-m}$ we start estimating $\phi_\varepsilon (X_\lambda(t))$, where, for any $\varepsilon>0$,
	\[ \phi_\varepsilon (y) \doteq (\norm{y}_{-1}^2 + \varepsilon^2)^{-\frac{1+m}{2}}. \]
	Notice that
	\begin{align*}
	D\phi_\varepsilon(y) & = (1-m)\left( \norm{y}_{-1}^2 + \varepsilon^2 \right)^{-\frac{1+m}{2}}y, \\
	D^2 \phi_\varepsilon (y) & = (1-m) \left( \norm{y}_{-1}^2 + \varepsilon \right)^{-\frac{1+m}{2}} - (1-m^2) \left( \norm{y}_{-1}^2 + \varepsilon^2 \right)^{-\frac{3+m}{2}} y \otimes y.
	\end{align*}
	By \ito's formula we have
	\[
	\begin{split} 
	d \phi_\varepsilon(X_\lambda (t)) & - (1-m)\left( \norm{X_\lambda(t)}_{-1}^2 + \varepsilon^2 \right)^{-\frac{1+m}{2}} \inner{ X_\lambda(t), \nu \Delta X_\lambda(t) + \Delta \psi_\lambda(X_\lambda(t)) }_{-1} dt \\
	= & \, (1-m)\left( \norm{X_\lambda(t)}_{-1}^2 + \varepsilon^2 \right)^{-\frac{1+m}{2}} \inner{X_\lambda(t), \frac{1}{2} \dive (\bba \bb \nabla X_\lambda(t)) }_{-1} dt \\
	& + \frac{1}{2} (1-m) \left( \norm{X_\lambda(t)}_{-1}^2 + \varepsilon \right)^{-\frac{1+m}{2}} \norm{\bb \nabla X_\lambda(t)}_{-1}^2 dt \\
	& - \frac{1}{2} (1-m^2) \left( \norm{X_\lambda(t)}_{-1}^2 + \varepsilon^2 \right)^{-\frac{3+m}{2}} \norm{X_\lambda(t)}_{-1}^2 \norm{\bb \nabla X_\lambda(t)}_{-1} dt \\
	& + (1-m) \left( \norm{X_\lambda(t)}_{-1}^2 + \varepsilon^2 \right)^{-\frac{1+m}{2}} \inner{X_\lambda(t),\bb \nabla X_\lambda(t) \cdot d \beta (t)}_{-1},
	\end{split} 
	\]
	which can be rewritten as
	\[
	\begin{split} 
	d \phi_\varepsilon(X_\lambda (t)) & + (1-m) \nu \left( \norm{X_\lambda(t)}_{-1}^2 + \varepsilon^2 \right)^{-\frac{1+m}{2}} \abs{X_\lambda(t)}_{2}^2 dt \\ 
	& + (1-m)\left( \norm{X_\lambda(t)}_{-1}^2 + \varepsilon^2 \right)^{-\frac{1+m}{2}} \inner{ X_\lambda(t), \psi_\lambda(X_\lambda(t)) }_{2} dt \\
	= & \, \frac{1}{2} (1-m)\left( \norm{X_\lambda(t)}_{-1}^2 + \varepsilon^2 \right)^{-\frac{1+m}{2}} \inner{X_\lambda(t), \dive (\bba \bb \nabla X_\lambda(t)) }_{-1} dt \\
	& + \frac{1}{2} (1-m) \left( \norm{X_\lambda(t)}_{-1}^2 + \varepsilon \right)^{-\frac{1+m}{2}} \norm{\bb \nabla X_\lambda(t)}_{-1}^2 dt \\
	& - \frac{1}{2} (1-m^2) \left( \norm{X_\lambda(t)}_{-1}^2 + \varepsilon^2 \right)^{-\frac{3+m}{2}} \norm{X_\lambda(t)}_{-1}^2 \norm{\bb \nabla X_\lambda(t)}_{-1} dt \\
	& + (1-m) \left( \norm{X_\lambda(t)}_{-1}^2 + \varepsilon^2 \right)^{-\frac{1+m}{2}} \inner{X_\lambda(t),\bb \nabla X_\lambda(t) \cdot d \beta (t)}_{-1}.
	\end{split} 
	\]
	Now noticing that $\inner{X_\lambda, \psi_\lambda(X_\lambda)}_2 = \rho \abs{X_\lambda}_{1+m}^{1+m}$ and estimating the terms on the right-hand side as we did in the proof of \Cref{thm:existence}, we get
	\[ 
	\begin{split} 
	d \phi_\varepsilon(X_\lambda (t)) & + (1-m) \nu \left( \norm{X_\lambda(t)}_{-1}^2 + \varepsilon^2 \right)^{-\frac{1+m}{2}} \abs{X_\lambda(t)}_{2}^2 dt \\ 
	& + (1-m) \rho \left( \norm{X_\lambda(t)}_{-1}^2 + \varepsilon^2 \right)^{-\frac{1+m}{2}} \abs{X_\lambda(t)}_{1+m}^{1+m} dt \\
	\le & \, \frac{1}{2} (1-m) \tilde{C} \gamma \left( \norm{X_\lambda(t)}_{-1}^2 + \varepsilon^2 \right)^{-\frac{1+m}{2}}  \abs{X_\lambda (t)}_2^2 dt \\
	& + \frac{1}{2} (1-m) \abs{\bb}_\infty^2 \left( \norm{X_\lambda(t)}_{-1}^2 + \varepsilon^2 \right)^{-\frac{1+m}{2}} \abs{X_\lambda(t)}_2^2 dt \\
	& + \frac{1}{2} (1-m) \abs{(\dive b_j)_j}_\infty^2 \left( \norm{X_\lambda(t)}_{-1}^2 + \varepsilon^2 \right)^{-\frac{1+m}{2}} \norm{X_\lambda(t)}_{-1}^2 dt \\
	& + (1-m) \left( \norm{X_\lambda(t)}_{-1}^2 + \varepsilon^2 \right)^{-\frac{1+m}{2}} \inner{X_\lambda(t),\bb \nabla X_\lambda(t) \cdot d \beta (t)}_{-1}.
	\end{split} 
	\]
	Let $C_1 = \nu - \frac{\tilde{C} \gamma + \abs{\bb}_\infty^2}{2}$ and $C_2=\abs{(\dive b_j)_j}_\infty^2$, then integrating with respect to time from $r$ to $t$, we have
	\[ 
	\begin{split} 
	\norm{X_\lambda (t)}_{-1}^{1-m} & + (1-m) C_1 \int_r^t \left( \norm{X_\lambda(s)}_{-1}^2 + \varepsilon^2 \right)^{-\frac{1+m}{2}} \abs{X_\lambda(s)}_{2}^2 ds \\ 
	& + (1-m) \rho \int_r^t \left( \norm{X_\lambda(s)}_{-1}^2 + \varepsilon^2 \right)^{-\frac{1+m}{2}} \abs{X_\lambda(s)}_{1+m}^{1+m} ds \\
	\le & \, \norm{X_\lambda (r)}_{-1}^{1-m} 
	+ \frac{C_2 (1-m)}{2} \int_r^t \left( \norm{X_\lambda(s)}_{-1}^2 + \varepsilon^2 \right)^{-\frac{1+m}{2}} \norm{X_\lambda(s)}_{-1}^2 ds \\
	& + (1-m) \int_r^t \left( \norm{X_\lambda(s)}_{-1}^2 + \varepsilon^2 \right)^{-\frac{1+m}{2}} \inner{X_\lambda(s),\bb \nabla X_\lambda(s) \cdot d \beta (s)}_{-1}.
	\end{split} 
	\]
	Taking $\varepsilon \to 0$ yields
	\[ 
	\begin{split} 
	\norm{X_\lambda (t)}_{-1}^{1-m} & + (1-m) C_1 \int_r^t \norm{X_\lambda(s)}_{-1}^{-1-m} \abs{X_\lambda(s)}_{2}^2 \mathbbm{1}_{\norm{X_\lambda(s)}_{-1} > 0}(s) \, ds \\ 
	& + (1-m) \rho \int_r^t \norm{X_\lambda(s)}_{-1}^{-1-m} \abs{X_\lambda(s)}_{1+m}^{1+m} \mathbbm{1}_{\norm{X_\lambda(s)}_{-1} > 0}(s) \, ds \\
	\le & \, \norm{X_\lambda (r)}_{-1}^{1-m} 
	+ \frac{C_2 (1-m)}{2} \int_r^t \norm{X_\lambda(s)}_{-1}^{1-m} ds \\
	& + (1-m) \int_r^t \norm{X_\lambda(s)}_{-1}^{-1-m} \inner{X_\lambda(s),\bb \nabla X_\lambda(s) \cdot \left( \mathbbm{1}_{\norm{X_\lambda(s)}_{-1} > 0}(s) \, d \beta (s) \right)}_{-1}.
	\end{split} 
	\]
	Set now $K_m = C_2(1-m)/2$. By the stochastic Gronwall's lemma we have
	\begin{multline*} 
	\begin{aligned} 
	\norm{X_\lambda (t)}_{-1}^{1-m} & + (1-m) C_1 \int_r^t e^{K_m (t-s)} \norm{X_\lambda(s)}_{-1}^{-1-m} \abs{X_\lambda(s)}_{2}^2 \mathbbm{1}_{\norm{X_\lambda(s)}_{-1} > 0}(s) \, ds \\ 
	&+ (1-m) \rho \int_r^t e^{K_m (t-s)} \norm{X_\lambda(s)}_{-1}^{-1-m} \abs{X_\lambda(s)}_{1+m}^{1+m} \mathbbm{1}_{\norm{X_\lambda(s)}_{-1} > 0}(s) \, ds 
	\end{aligned} \\
	\begin{aligned} 
	\le & \, e^{K_m (t-r)} \norm{X_\lambda (r)}_{-1}^{1-m} + (1-m) \int_r^t e^{K_m (t-s)} \\
	& \norm{X_\lambda(s)}_{-1}^{-1-m} \inner{X_\lambda(s),\bb \nabla X_\lambda(s) \cdot \left( \mathbbm{1}_{\norm{X_\lambda(s)}_{-1} > 0}(s) \, d \beta (s) \right)}_{-1}.
	\end{aligned} 
	\end{multline*} 
	Recalling that $C_1 > 0$ because of \eqref{eq:hp-on-b-nu}, we get \eqref{eq:extinction-estimate}.
\end{proof}

For the next step, we need an additional assumption concerning $d$ and $m$. 
In particular, we would like to have a constant $C_m \ge 0$ such that
\begin{equation} \label{eq:sobolev-inequality}
\norm{y}_{-1}^{-1-m} \abs{y}_{1+m}^{1+m} \ge C_m, \quad \text{for every } y \in H^{-1}, 
\end{equation} 
which is equivalent to have 
\[ L^{1+m}(\cO) \subset H^{-1}(\cO), \quad \text{with continuous embedding}. \]
By duality, this is equivalent to
\[ H^1_0(\cO) \subset L^{\frac{1+m}{m}}(\cO), \quad \text{with continuous embedding}, \]
hence, by Sobolev embedding Theorem, we have that \eqref{eq:sobolev-inequality} holds provided
\[ \frac{m}{m+1} > \frac{1}{2} - \frac{1}{d}, \]
which explains the meaning of hypothesis \eqref{eq:hp-on-m-d}.

We are now ready to prove our extinction in finite time result.
\begin{thm}\label{thm:extinction-fast-diffusion}
	Suppose \Cref{ass:extinction} holds. 
	Let $X$ be a solution to \cref{eq:spde ito} and set $\tau_x \doteq \inf \{ t > 0 \colon X(t)=0\}$, then $X(t,\omega) = 0$ for every $t > \tau_x(\omega)$.
	Moreover, the extinction probability is finite and
	\begin{equation}\label{eq:extinction-prob}
	\Pr (\tau_x > t) \le \frac{K_m \norm{x}_{-1}^{1-m}}{\rho \, C_m \, (1-m) \left( 1 - e^{-K_m t} \right)}.
	\end{equation}
\end{thm}

\begin{proof}
	By \eqref{eq:extinction-estimate}, taking into account \eqref{eq:sobolev-inequality}, we have for every $t \ge r \ge 0$
	\begin{multline*}
	\norm{X_\lambda (t)}_{-1}^{1-m} + C_m (1-m) \rho \int_r^t e^{K_m (t-s)} \mathbbm{1}_{\norm{X_\lambda(s)}_{-1} > 0}(s) \, ds 
	\le  \, e^{K_m (t-r)} \norm{X_\lambda (r)}_{-1}^{1-m} \\
	+ (1-m) \int_r^t e^{K_m (t-s)} \norm{X_\lambda(s)}_{-1}^{-m-1} \inner{X_\lambda(s), \bb \nabla X_\lambda(s) \cdot \left( \mathbbm{1}_{\norm{X_\lambda(s)}_{-1} > 0}(s) \, d \beta (s) \right)}_{-1}.
	\end{multline*}
	Let $\lambda \to 0$ to find
	\begin{multline*}
	\norm{X (t)}_{-1}^{1-m} + C_m (1-m) \rho \int_r^t e^{K_m (t-s)} \mathbbm{1}_{\norm{X (s)}_{-1} > 0}(s) \, ds 
	\le  \, e^{K_m (t-r)} \norm{X (r)}_{-1}^{1-m} \\
	+ (1-m) \int_r^t e^{K_m (t-s)} \norm{X(s)}_{-1}^{-m-1} \inner{X(s), \bb \nabla X(s) \cdot \left( \mathbbm{1}_{\norm{X(s)}_{-1} > 0}(s) \, d \beta (s) \right)}_{-1},
	\end{multline*}
	which can be equivalently written as
	\begin{multline}\label{eq:extinction-proof}
	e^{-K_m t} \norm{X (t)}_{-1}^{1-m} + C_m (1-m) \rho \int_r^t e^{-K_m s} \mathbbm{1}_{\norm{X (s)}_{-1} > 0}(s) \, ds 
	\le  \, e^{-K_m r} \norm{X (r)}_{-1}^{1-m} \\
	+ (1-m) \int_r^t e^{-K_m s} \norm{X(s)}_{-1}^{-m-1} \inner{X(s), \bb \nabla X(s) \cdot \left( \mathbbm{1}_{\norm{X(s)}_{-1} > 0}(s) \, d \beta (s) \right)}_{-1}.
	\end{multline}
	Recalling that $K_m = C_2 (1-m) / 2$ and defining $Y(t) \doteq e^{-C_2 t/2} X(t)$, this proves that the process $\norm{Y(t)}_{-1}^{1-m}$ is a nonnegative supermartingale, i.e.,
	\[ \E \left[ \norm{Y(t)}_{-1}^{1-m} \vert \, \cF_r \right] \le \norm{Y(r)}_{-1}^{1-m}, \quad \text{for every } t \ge r. \]
	This implies, for any couple of stopping times $\tau_1$ and $\tau_2$, that
	\[ \tau_1 > \tau_2 \, \Rightarrow \, \norm{Y(\tau_1)}_{-1} \le \norm{Y(\tau_2)}_{-1}, \]
	and, in particular, for any $t > \tau_x = \inf\{t > 0\colon X(t) = 0\} $, we have
	\[ \norm{Y(t)}_{-1} \le \norm{Y(\tau_x)}_{-1} = 0, \]
	that is,
	\[ \norm{X(t)}_{-1} = \norm{X(\tau_x)}_{-1} = 0, \quad \Pr\text{-a.s.} \]
	
	Now set $r = 0$ in \eqref{eq:extinction-proof} and take the expectation to get
	\[ \E \norm{Y(t)}_{-1}^{1-m} + C_m (1-m) \rho \int_0^t e^{-K_m s} \, \Pr (\tau_x > s) \, ds \le \norm{x}_{-1}^{1-m}, \]
	which gives 
	\[ \E \norm{Y(t)}_{-1}^{1-m} + C_m (1-m) \rho \int_0^t e^{-K_m s} \, \Pr (\tau_x > t) \, ds \le \norm{x}_{-1}^{1-m}, \]
	and \eqref{eq:extinction-prob} follows.
\end{proof}

\section{Self-organized criticality model} \label{sec:soc}

In this section we will introduce the \emph{self-organized criticality} (SOC) model, which is a special case of porous media equation with
\[ \psi(r) = \rho \sign r + \phi(r), \quad \text{for all } r \in \R, \]
where $\rho > 0$, $\phi$ a maximal monotone graph in $\R \times \R$ and
\[ \sign r = \begin{cases}
\dfrac{r}{\abs{r}}, & \text{for } r \ne 0, \\[6pt]
\{ r \in \R; \abs{r} \le 1 \}, & \text{for } r = 0.
\end{cases} \]
Such a choice of $\psi$ represents the so-called \emph{sand-pile model} or \emph{Bak-Tang-Wiesenfeld model}. 
The deterministic version of the sand-pile model was first introduced in~\cite{bak1987self,bak1988selforganized}, while its stochastic counterpart has been studied, e.g., in~\cite{barbu2010selforganized,barbu2009stochastic,barbu2012stochastic}.

In the following we formalize the deterministic model referring to the method treated in~\cite{bantay1992selforganization}.
Let $\cO$ be an $N \times N$ discrete region of points, we label each of those points with an integer index $i \in \{1,\ldots,N^2\}$. 
We associate a height, $X_i(t)$, to every index $i$ at a certain time $t$.
Now, we can select, randomly, a site $i$ and increase $X_i(t)$ by $1$, leaving the other sites unchanged.
A \emph{toppling} event occurs if the height at a site exceed of a given \emph{critical value} $X_c$. A site whose height is greater than $X_c$ is called \emph{activated site}.

If $X_i(t) > X_c$, then
\begin{equation} \label{eq: btw model discrete point}
X_j(t+1) \to X_j(t) - Z_{ij}, \quad j = 1,\ldots,N^2, 
\end{equation} 
where $Z = (Z_{ij})_{ij}$ is a $N^2 \times N^2$ matrix such that
\[ Z_{ij} =
\begin{cases}
4, & \text{if } i = j, \\
-1, & \text{if } i \text{ and } j \text{ nearest neighbours}, \\
0, & \text{otherwise}.
\end{cases}
\]
Consider $X(t) = (X_i(t))_i$, then, the dynamic of the system can be written, starting from equation~\eqref{eq: btw model discrete point}, as
\begin{equation}\label{eq: btw model discrete}
X(t+1) = X(t)-Z f(t),
\end{equation}
where $f(t) = (f_i(t))_i = \bigl( H(X_i(t)-X_c) \bigr)_i$ and $H$ is the \emph{Heaviside function} defined as
\[ H(x) = \begin{cases}
1 & \text{if } x \ge 0,\\
0 & \text{if } x < 0.
\end{cases} \]
Noticing that the matrix $Z$ is a discretized version of the Laplace operator $\Delta$, we can claim that equation~\eqref{eq: btw model discrete} is the discrete version of the following partial differential equation for $X\colon \interval[open]{0}{+\infty} \times \cO \to \R$,
\begin{equation} \label{eq:soc pde}
\frac{\partial X}{\partial t}(t) = \Delta H(X(t)-X_c), \quad (t,\xi) \in \interval[open]{0}{+\infty} \times \cO, 
\end{equation} 
where $\cO \subset \R^2$ is a continuous spatial domain, $\Delta$ is the $2$-dimensional Laplace operator and $H$ is the Heaviside function.
More generally, one can consider $\cO \subset \R^d$, $d = 1,2,3$, and replace $H$ by a continuous function with jump at~$0$.
One has to associate to equation~\eqref{eq:soc pde} an initial value condition 
\begin{equation*}
X(0,\xi) = X_0(\xi), \quad \xi \in \cO,
\end{equation*}
with $X_0 \colon \cO \to \R$ representing the initial configuration of the system, and boundary conditions on $\partial \cO$, a common one being the Dirichlet condition
\begin{equation*}
X(t,\xi) = 0, \quad (t,\xi) \in \interval[open]{0}{+\infty} \times \partial \cO.
\end{equation*}

We want now to treat \cref{eq:spde ito} in the particular case of self-organized criticality. 
If we consider 
\[ \psi(r) = \rho \sign (r), \quad r \in \R, \]
where $\rho > 0$, then
\Cref{thm:existence} still applies, since $\psi$ satisfies the hypotheses therein.

\Cref{thm:extinction-fast-diffusion} holds with the same proof in the case $m=0$, which is exactly the case of self-organized criticality, since
\[ \psi(r) = \rho \abs{r}^{m-1} r = \rho \abs{r}^m \sign(r). \] 
However, under this assumption, condition \eqref{eq:hp-on-m-d} imposes $d = 1$.

\section{Concluding remarks} \label{sec:conclusion}

\Cref{thm:existence} provides a nice existence result for \cref{eq:spde strato}, however, as we pointed out in \Cref{sec:framework}, hypothesis \eqref{eq:hp-on-b-nu} is quite restrictive, since it imposes in particular that $\nu > 0$ to avoid the loss of the noise term. 
It could be interesting to see if it is possible to gain existence even with $\nu = 0$, but keeping the noise.

As regards the asymptotic behaviour of solutions, \Cref{thm:extinction-fast-diffusion} ensures extinction in finite time of the solution to the fast diffusion model, while for the SOC model, in \Cref{sec:soc}, we only have the result for $d =1$.
One may wonder what happens in the case $d \ge 2$, does extinction in finite time phenomenon still take place?
Some asymptotic results for the case of SOC in stochastic porous media equations of the type
\[ dX - \Delta \psi(X) \, dt = \sigma (X) \, dW, \]
have been provided by V.~Barbu, G.~Da Prato, and M.~R\"ockner in~\cite[Ch.~3.8]{barbu2016stochastic} as well as B.~Gess in~\cite{gess2015finite}, the latter guaranteeing, under some suitable assumptions, the extinction in finite time of solutions also for $d > 1$.
However, in the case of Stratonovich gradient noise, what happens for $d>1$ is still to be proved, up to our knowledge, and it could be the next step to be tackled in future works.

\section*{Acknowledgement}

This work has been carried out during my master's degree studies at the University of Verona. 
I am sincerely grateful to Luca Di Persio for having introduced me to the problem and helped me through its resolution.

\bibliographystyle{hplain}
\bibliography{soc.bib} 

\end{document}